\documentclass[a4,11pt]{nmd/article}
\usepackage{enumitem}
\usepackage{graphicx,color}
\usepackage{epsfig,manfnt}
\usepackage{amsfonts}
\usepackage{amssymb}
\usepackage{amsmath}
\usepackage{amsthm}
\usepackage{latexsym}
\usepackage{amscd}
\usepackage{flafter}
\usepackage{epsf}
\usepackage{epstopdf}
\usepackage{inputenc}
\usepackage{stmaryrd}
\usepackage{tikz,tikz-cd}
\usepackage{mathrsfs}
\usepackage{lipsum}
\newcommand{\fmap}{{{f}}}

\newcommand{\x}{\mathrm{\bf{x}}}

\newcommand{\Sig}{\Sigma}
\newcommand{\ra}{\rightarrow}

\newcommand{\Dbb}{\mathbb{D}}

\newcommand{\Acal}{\mathcal{A}}

\newcommand{\Dcal}{\mathcal{D}}

\newcommand{\Qcal}{\mathcal{Q}}

\newtheorem{thm}{Theorem}[section]
\newtheorem{prop}[thm]{Proposition}
\newtheorem{cor}[thm]{Corollary}

\newtheorem{lem}[thm]{Lemma}

\theoremstyle{definition}

\newtheorem{defn}[thm]{Definition}

\newcommand{\thetabarp}{{\tilde{\theta}}}
\newcommand{\uf}{{\mathfrak{l}}}
\newcommand{\A}{{\mathbb{A}}}
\newcommand{\tf}{{\mathfrak{t}_q}}

\newcommand{\bs}{{\vec{b}}}
\newcommand{\ns}{{\vec{n}}}
\newcommand{\ms}{{\vec{m}}}
\newcommand{\zs}{{\vec{z}}}
\newcommand{\ws}{{\vec{w}}}
\newcommand{\ps}{{\vec{p}}}
\newcommand{\qs}{{\vec{q}}}
\newcommand{\torsion}{{\mathfrak{t}_q'}}
\newcommand{\h}{\mathbf{h}}

\newcommand{\var}{\mathsf{u}}
\newcommand{\varw}{\mathsf{w}}
\newcommand{\map}{\mathfrak{f}}
\newcommand{\alphas}{{\vec{\alpha}}}
\newcommand{\betas}{{\vec{\beta}}}
\newcommand{\seq}{{\mathbf{a}}}
\newcommand{\seqb}{{\mathbf{b}}}
\newcommand{\seqc}{{\mathbf{c}}}
\newcommand{\gammas}{\vec{\gamma}}

\newcommand{\y}{\mathbf{y}}

\newcommand{\CFT}{\mathrm{CF}}
\newcommand{\HFT}{\mathrm{HF}}

\begin{document}
\title{Rational tangle replacements\\ and knot Floer homology}%

\author{Eaman Eftekhary}%
\address{School of Mathematics, Institute for Research in 
Fundamental Sciences (IPM), P. O. Box 19395-5746, Tehran, Iran}%
\email{eaman@ipm.ir}
\begin{abstract}
From the link Floer complex of a link $K$, we extract a lower bound $\torsion(K)$ for the rational unknotting number of $K$ (i.e. the minimum number of rational replacements required to unknot $K$). Moreover, we show that the torsion obstruction $\tf(K)=\widehat{\mathfrak{t}}(K)$ from an earlier paper of Alishahi and the author is a lower bound for the proper rational unknotting number. Moreover,  $\tf(K\#K')=\max\{\tf(K),\tf(K')\}$ and $\torsion(K\#K')=\max\{\torsion(K),\torsion(K')\}$. For the torus knot $K=T_{p,pk+1}$ we compute $\torsion(K)=\lfloor p/2\rfloor$ and $\tf(K)=p-1$. 
\end{abstract}

\maketitle

\section{Introduction} 
As a byproduct of the study of cobordism maps in Heegaard Floer theory of tangles, Alishahi and the author introduced lower bounds for the Gordian distance $u(K,K')$ of a pair of knots $K,K'$, and in particular the unknotting number $u(K)$ of $K$ \cite{AE-cobordisms-v1}, which developed into the independent paper \cite{AE-unknotting}. Alishahi applied the strategy of  \cite{AE-unknotting}  to bound $u(K,K')$ using Khovanov homology \cite{Alishahi-unknotting}. Her work was followed by other lower bounds with roots in Khovanov homology (c.f.  \cite{AD-unknotting} and  \cite{Cetal-unknotting}). Recently, Iltgen, Lewark and Marino proved that their invariant $\lambda$ \cite{ILM-rational},  the best known unknotting bound from Khovanov homology,  is in fact a lower bound for the {\emph{proper rational distance}} $u_{q}$,  defined as follows:
\begin{defn}\label{defn:rational-unknotting}
The oriented links $K$ and $K'$ are related by a {\emph{rational replacement}} (or an RR) if  replacing a rational tangle $T$ in $K$ with another rational tangle $T'$ gives $K'$. The replacement is called an {\emph{orientation-preserving}} rational replacement (or an ORR) if it respects the orientations, and  is called a {\emph{proper}} rational replacement (or a PRR) if the arcs of $T$ and $T'$ connect the same tangle end points. The {\emph{rational distance}} $u'_{q}(K,K')$ (resp.  the {\emph{PR-distance}} $u_q(K,K')$ and the {\emph{OR-distance}} $u''_{q}(K,K')$) is defined as the minimum number of RRs (resp. PRRs and ORRs) required to change $K$ to $K'$. The {\emph{rational unknotting number}} $u_q'(K)$,the {\emph{PR-unknotting number}} $u_q(K)$ and the {\emph{OR-unknotting number}} $u_q''(K)$ of $K$ are defined as the minimum  rational distance, PR-distance  and OR-distance of $K$ from an unlink, respectively.  
\end{defn} 

Rational unknotting was considered by Lines \cite{lines} and McCoy \cite{McCoy}. More, recently, McCoy and Zenter adapted the so called Montesinos trick, to study proper rational unknotting as well \cite{MZ-rr}. The work of Iltgen, Lewark and Marino \cite{ILM-rational} is the first connection between (proper) rational unknotting and the quantum invariants.\\

 In this paper, we use link Floer homology to bound rational distance and OR-distance (and thus, PR-distance) from below. Let $K$ be an oriented link and $\ps$ denote  a {\emph{marking}} of $K$, i.e. a collection of $|\ps|$ marked points on $K$ which includes at least one marked point on each connected component. Let $\CFT(K,\ps)$ denote the link chain complex for  $(K,\ps)$, constructed from the Heegaard diagram $(\Sigma,\alphas,\betas,\zs,\ws)$. Thus, $\zs$ and $\ws$ are collections of $|\ps|$ marked points on $\Sig$ which correspond to $\ps$ (c.f. \cite{AE-tangles} and \cite{OS-link}). $\CFT(K,\ps)$ is generated over $\F[\var,\varw]$ by the intersection points in $\T_\alphas\cap\T_\betas$ (where $\F=\Z/2$) and is equipped with the differential 
 \begin{align*}
d(\x)=\sum_{\y\in\T_\alphas\cap \T_{\betas}}\ \ \sum_{\substack{\phi\in\pi_2(\x,\y)\\ \mu(\phi)=1}}\#(\Mhat(\phi))\var^{n_\zs(\phi)}\varw^{n_\ws(\phi)}\cdot \y.
\end{align*}
For a $\F[\var,\varw]$-algebra $\A$,  $\HFT(K,\ps,\A)$ denotes the homology of  $\CFT(K,\ps,\A):=\CFT(K,\ps)\otimes_{\F[\var,\varw]}\A$. The pseudometric $\uf_\A(K,K')$ between the oriented links $K$ and $K'$  is defined as the least  $k\in\Z$ such that there are markings $\ps$ of $K$ and $\ps'$  of $K'$ with $|\ps|=|\ps'|$ and chain maps 
\begin{align*}
&\fmap:\CFT(K,\ps,\A)\ra \CFT(K',\ps',\A)\quad\text{and}\quad\fmap':\CFT(K',\ps',\A)\ra \CFT(K,\ps,\A),
\end{align*}
with $\fmap\circ\fmap'$ and $\fmap'\circ\fmap$ chain homotopic to $\var^k$. Throughout the paper, set $\A=\A'=\F[\var]$, while the action of $\varw$ on $\A$ and $\A'$ is defined as multiplication by $0$ and $\var$, respectively. 
We denote $\uf_\A(K,K')$ and $\uf_{\A'}(K,K')$ by $\tf(K,K')$ and $\torsion(K,K')$, respectively. If $u(K,K')$ denotes the Gordian distance between $K$ and $K'$, Alishahi and the author prove (see \cite{AE-unknotting}): 
  \[\tf(K,K'),\torsion(K,K')\leq \uf_{\F[\var,\varw]}(K,K')\leq u(K,K').\]

\begin{thm}\label{thm:main-introduction}
Given the oriented links $K,K'$ we have 
\[\tf(K,K')\leq u''_{q}(K,K')\leq u_{q}(K,K')\quad\quad\text{and}\quad\quad\torsion(K,K')\leq u'_q(K,K').\] 
In particular, $\tf(K):=\tf(K,U)\leq u_{q}(K)$ and $\torsion(K):=\torsion(K,U)\leq u'_{q}(K)$, where $U$ is the unknot.
\end{thm}

Given relatively prime integers $1<p<q$, write $(p,q)\leadsto (p-2i,q-2j)$, where  $0<i\leq p/2$ and $j$ are chosen so  that $iq=jp\pm 1$. Let $k(p,q)$ denote the smallest $k\in\Z^+$ so that 
\[(p,q)=(p_0,q_0)\leadsto (p_1,q_1)\leadsto (p_2,q_2)\leadsto \cdots\leadsto (p_k,q_k),\]
with $p_k\in\{0,1\}$. In particular, $k(p,q)\leq \lfloor p/2\rfloor$ for all  $1<p<q$, $k(p,pn+1)=\lfloor p/2\rfloor$ for all $p>1$ and $k(p,pn+2)=1$ for odd values of $p>1$.

\begin{thm}\label{thm:torus-knots}
If $1<p<q$ are relatively prime integers we have 
\begin{align*}
&\tf(T_{p,q})=p-1\leq u''_{q}(T_{p,q})\leq u_{q}(T_{p,q})\quad\quad\text{and}\quad\quad\torsion(p,q)\leq u'_q(T_{p,q})\leq k(p,q).
\end{align*}
Moreover, given the integers $p>1$ and $n>0$, we have
\begin{align*}
\torsion (T_{p,pn+1})= u'_q(T_{p,pn+1})=k(p,pn+1)=\left\lfloor \frac{p}{2}\right\rfloor.
\end{align*}
\end{thm}
In fact, in all our computations the equality $\torsion(T_{p,q})= u'_q(T_{p,q})= k(p,q)$ is satisfied. On the other hand, the computation of the invariants for knots with at most $10$ crossings implies that $\tf(K)=\torsion(K)=1$, unless $K$ is one of the knots $8_{19}$,  $10_{124}$, $10_{128}$, $10_{139}$, $10_{152}$, $10_{154}$ and $10_{161}$. If $K$ is any of these latter $7$ knots we have $\tf(K)=2$.\\

Given a knot $K$ (i.e. under the assumption that $K$ has one component), let  
\begin{align*}
Q_K(q,t)&:=\sum_{i,j} \mathrm{dim}\left(\widehat{\mathrm{HFK}}_j(K,i)\right)\cdot q^jt^i
\end{align*}
denote the  polynomial  encoding the rank of $\widehat{\mathrm{HFK}}(K)$ in different Alexander and homological gradings. $K$ has {\emph{thin knot Floer homology}} if $q^iQ_K(q,t)$ is a polynomial in $qt$ for some integer $i$.   Denote the set of all knots with thin knot Floer homology  by $\Qcal$. In particular, $K\in\Qcal$  if $K$ is alternating, or even quasi-alternating, by \cite{OS-alternating} and \cite{MO-quasi}. Since $\tf(K)=\torsion(K)=1$ for every $K\in\Qcal$, the following corollary follows from Theorem~\ref{thm:main-introduction}:

\begin{cor}\label{cor:quasi-alternating}
 Given an arbitrary knot $K$, 
\[u_q(K,\Qcal)\geq u''_q(K,\Qcal)\geq \tf(K)-1\quad\text{and}\quad u'_q(K,\Qcal)\geq \torsion(K)-1.\]
 In particular,  $T_{p,q}$ may not be changed to a quasi-alternating knot with less than $p-2$ PRRs. 
\end{cor}
We also obtain the following obstruction for unknotting a knot with a single PRR, which is useful since  $\tau$ and $Q_K$ are known for knots with few crossings (e.g. see \cite{BG-computations}).
\begin{prop}\label{prop:Alexander-polynomial}
$u''_q(K)>1$  unless $Q_K(q,t)-t^{\tau(K)}$ is divisible by $1+qt$.
\end{prop}

Since $\CFT(K\#K')=\CFT(K)\otimes_{\F[\var,\varw]}\CFT(K')$, we  obtain the following connected sum formula:

\begin{prop}\label{prop:connected-sum}
For every two knots $K$ and $K'$ we have
\[\tf(K\#K')=\max\big\{\tf(K),\tf(K')\big\}\quad\quad\text{and}\quad\quad \torsion(K\#K')=\max\big\{\torsion(K),\torsion(K')\big\}.\]
\end{prop} 

{\bf{Acknowledgments.}} The author would like to thank Lukas Lewark for  bringing up the potential connection between torsion invariants from knot  Floer homology and rational distance of knots, and helpful discussions.

\section{Heegaard triples for rational tangle replacements}\label{sec:heegaard-triple}
Let us  assume that $K\subset \R^3$ is an oriented knot or link and the marking $\ps$ on $K$ is also fixed. We further assume that the intersection of a ball $B\subset \R^3$ with $K$ is the trivial $2$-tangle $T$ and that under the projection $\pi:\R^3\ra \R^2$  over the $xy$-plane $\R^2$  (from a fixed point in $\R^3\setminus \R^2$), the image of $K\subset \R^3$ gives a knot diagram for $K$. Moreover,  the image of $(B,K\cap B)$ is $(D,\pi(K)\cap D)$, where $D$ is a disk and $\pi(K)\cap D$ is a pair of disjoint line segments. Let $J$ denote a line segment in $D$ which connects the two line segments in $\pi(K)\cap D$. For simplicity, we assume that $\qs=\pi(\ps)$ includes at least one marked point on each line segment connecting two self-intersections of $\pi(K)$. We also fix a marked point $q'$ on $J$ and a distinguished marked point $q''\in\qs$. Let $\overline{K}$ denote the union of $\pi(K)$ with $J$, and $\Sigma$ denote the boundary of the $\epsilon$-neighborhood of $\overline{K}\subset \R^3$ for a sufficiently small value of $\epsilon>0$. $\Sigma$ is a closed surface of genus $g+1$ if $\pi(K)$  has $g-1$ self-intersections. The intersection of $\Sig$ with the plane $\R^2$ is a $1$-manifold, which is a collection of $g+2$ circles $\alpha_{-1},\alpha_0,\ldots,\alpha_{g}$ in $\R^2$. We choose the labels so that $\alpha_{-1}$ is the distinguished circle which includes all other circles in its interior (as a curve on $\R^2$).\\

\begin{figure}
\def\svgwidth{0.9\textwidth}
\begin{center}
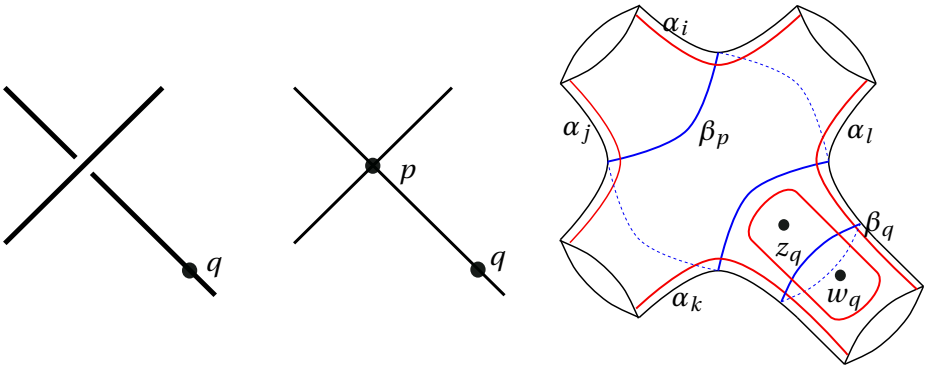
\caption{\label{fig:crossing} A crossing in a knot diagram for $K$ (left) corresponds to a self intersection $p$ of $\overline{K}$ (middle). The ball of radius $5\epsilon$ around $p$ cuts $\Sigma_0$ in a subsurface (right). Associated with $p$ is $\beta_p$ and associated with the marked point $q$ is a meridian $\beta_q$.}
\end{center}
\end{figure} 
 
Each crossing in the diagram of $K$, as illustrated in Figure~\ref{fig:crossing} (left), corresponds to a self intersection $p$ of $\overline{K}$. The intersection of the ball of radius $5\epsilon$ around $p$ with $\Sigma$ is then a sphere with $4$ disks removed, as illustrated in Figure~\ref{fig:crossing} (right). Associated with each such crossing, we may then introduce a simple closed curve $\beta_p$, which is included on the aforementioned (punctured) sphere, as illustrated in Figure~\ref{fig:crossing} (right). Each such $\beta$-curve has $4$ intersections with some $\alpha$-curves $\alpha_i,\alpha_j,\alpha_k$ and $\alpha_l$. Note that $i,j,k$ and $l$ are not necessarily different, and some of them may be equal to $-1$. Associated with the $g-1$ crossings $p_1,\ldots,p_{g-1}$ of the knot diagram for $K$, we thus obtain the $\beta$-curves $\beta_1=\beta_{p_1},\ldots,\beta_{g-1}=\beta_{p_{g-1}}$. Moreover, associated every $q\in\qs$, and also associated with $q'$, we obtain the simple closed meridians $\{\beta_q\}_{q\in\qs}$ and $\beta_0=\beta_{q'}$ on $\Sigma$  (see  Figure~\ref{fig:crossing}). For each $q\in\qs$, we place a pair of marked points $z_q$ and $w_q$ on $\Sigma$ on the two sides of $\beta_q$, so that traversing $K$ in the direction determined by its orientation determines a small arc from $w_q$ to $z_q$. Let $\alpha_q$ denote a small circle on $\Sigma$ which bounds a disk that contains $w_q$ and $z_q$. The  diagram $H_0=\big(\Sigma,\alphas,\betas,\zs,\ws\big)$ is then a Heegaard diagram representing the pointed link $(K,\ps)$, where $\zs=\big\{z_q\ \big|\ q\in\qs\big\} $, $\ws=\big\{w_q\ |\ q\in\qs\big\} $ and 
\begin{align*}
&\alphas=\big\{\alpha_0,\ldots,\alpha_g\big\}\cup\big\{\alpha_q\ |\ q\in\qs\ \text{and}\ q\neq q'\big\}\quad\text{and}\quad\betas=\big\{\beta_0,\ldots,\beta_{g-1}\big\}\cup\big\{\beta_q\ \big|\ q\in\qs\big\}.
\end{align*}

\begin{figure}
\def\svgwidth{0.7\textwidth}
\begin{center}
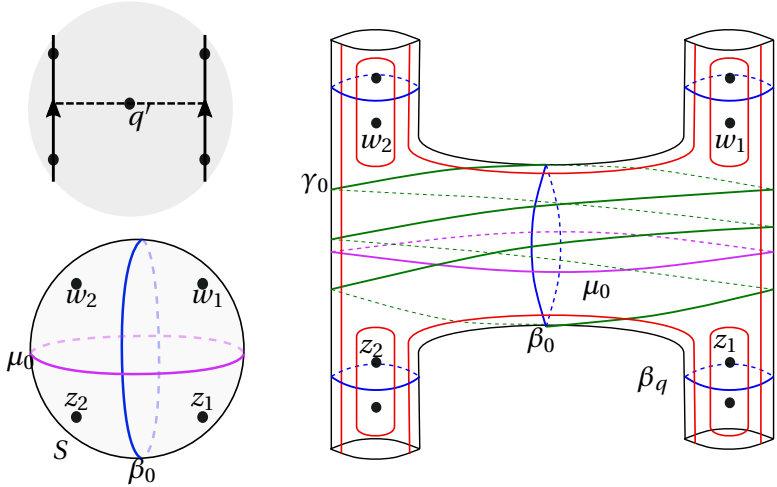
\caption{\label{fig:tangle-diagram} The intersection of $\Sigma$ with the $2\epsilon$-neighborhood of $D$ is illustrated. Changing the tangle $K\cap B$ to another rational tangle corresponds to applying a sequence of (Dehn) half twists along $\mu$ and $\beta_0$ to $\beta_0$ or $\mu_0$, which gives the simple closed curve $\gamma_0$.}
\end{center}
\end{figure}

Consider the intersection $\Sigma_0$ of $\Sigma$ with the $2\epsilon$-neighborhood of the disk $D\subset \R^2$ in $\R^3$, as  illustrated in Figure~\ref{fig:tangle-diagram} (top-left and right). Let $\mu_0$ denote a simple closed curve on $\Sigma$ which projects to (an $\epsilon$-extension) of $J$ under $\pi$. Denote the (Dehn) half twist along $\beta_0$ by $\map_0$ and the half twist along $\mu_0$ by $\map_1$. Given any sequence of integers $\seq=(a_1,b_1,\ldots,a_k,b_k)$, set
\[\map=\map_\seq=\map_0^{a_1}\circ \map_1^{b_1}\circ \map_0^{a_2}\circ \map_1^{b_2}\circ\cdots\circ \map_0^{a_k}\circ \map_1^{b_k},\quad \gamma_0=\gamma_\seq=\map(\beta_0)\quad\text{and}\quad\mu_0=\mu_\seq=\map(\mu_0).\]
For instance, the green simple closed curve in Figure~\ref{fig:tangle-diagram} illustrates $\gamma_0=\map_1^3(\beta_0)$. Let $\gamma_\bullet$ denote a Hamiltonian isotope of $\beta_\bullet$ for $\bullet=1,\ldots g-1$ or $\bullet\in\qs$ and set 
\[\gammas_\seq=\big\{\gamma_0,\ldots,\gamma_{g-1}\big\}\cup\big\{\gamma_q\ \big|\ q\in\qs\big\}.\] 
We choose a different generic Hamiltonian isotopy for each sequence $\seq$. Associated with $\map_0$ and $\map_1$ are the vertical and the horizontal half twists $\map_0'$ and $\map_1'$ which  may be applied to $(B,K\cap B)$. Correspondingly, the sequence $\seq$ also determined the diffeomorphism $\map'_\seq:B\ra B$ which preserves $K\cap \partial B$ and takes $K\cap B$ to a rational tangle $T=T_\seq\subset B$. The Heegaard diagram
\[H_{\seq}=(\Sig,\alphas,\gammas,\zs,\ws)\]
then represents the pointed link $(K_\seq=(K\setminus B)\cup T_\seq,\ps)$ obtained by replacing $K\cap B$ with the rational tangle $T_\seq$, provided that this rational replacement is orientation-preserving. If the RR is not orientation-preserving, the diagram $H'_{\seq}=(\Sig,\alphas,\gammas,\zs\cup\ws)$ represents the unoriented link $(K_\seq=(K\setminus B)\cup T_\seq,\ps)$. One should also note that every RR (for $K\cap B$) is obtained in this way or by doing the same procedure with $\gamma_0$ replaced with $\mu_0$ (see \cite{Conway-tangles} or \cite{Kauffman-tangles}). The latter case (where we use $\mu_0$ instead of $\gamma_0$) may be handled in a completely similar manner, and will not be discussed below. Associated with the Heegaard diagrams $H_\seq$ (in the case where $\seq$ corresponds to an ORR) and $H'_\seq$  we then obtain  the chain complexes 
\[(C_\seq,d_\seq)=\CFT(K_\seq,\ps)\otimes_{\F[\var,\varw]}\A\quad\text{and}\quad (C'_\seq,d'_\seq)=\CFT(H'_\seq)=\CFT(K_\seq,\ps)\otimes_{\F[\var,\varw]}\A'.\]
   The homology group $\H_\seq$ of $(C_\seq,d_\seq)$ and the homology $\H'_\seq$ of $(C'_\seq,d'_\seq)$ are then modules over $\F[\var]$. For $\seq=0=(0,0)$, we denote $C_\seq$, $\H_\seq$, $C'_\seq$ and $\H'_\seq$  by $C_{K,\ps}=C_0$, $\H_{K,\ps}=\H_0$, $C'_{K,\ps}=C'_0$, and $\H'_{K,\ps}=\H'_0$, respectively. Associated with the  RR $K_{\seq}\leadsto K_{\seqb}$, we obtain the Heegaard triple
 \[H_{\seq,\seqb}=(\Sigma,\alphas,\gammas_\seq,\gammas_{\seqb},\zs,\ws)\quad\text{and}\quad H'_{\seq,\seqb}=(\Sigma,\alphas,\gammas_\seq,\gammas_{\seqb},\zs\cup\ws).\]
In using the diagram $H_{\seq,\seqb}$, we implicitly assume that $K_\seq\leadsto K_\seqb$  is orientation-preserving.  The  Heegaard diagrams  $(\Sig,\gammas_\seq,\gammas_\seqb,\zs,\ws)$ and $(\Sig,\gammas_\seq,\gammas_\seqb,\zs\cup\ws)$  determine the chain complexes 
\[C_{\seq,\seqb}=\CFT(\Sig,\gammas_\seq,\gammas_\seqb,\zs,\ws)\otimes_{\F[\var,\varw]}\A\quad\text{and}\quad C'_{\seq,\seqb}=\CFT(\Sig,\gammas_\seq,\gammas_\seqb,\zs\cup\ws),\] 
i.e. the variable associated with $\ws$ is set equal to $0$ in $C_{\seq,\seqb}$, and is set equal to $\var$ in $C'_{\seq,\seqb}$. Again, the homology of $C_{\seq,\seqb}$ is denoted by $\H_{\seq,\seqb}$ and the homology of $C'_{\seq,\seqb}$ is denoted by $\H'_{\seq,\seqb}$. The holomorphic triangle maps give the $\F[\var]$-homomorphisms
 \begin{align*}
 &\Phi_{\seq,\seqb}:\H_\seq\otimes \H_{\seq,\seqb}\ra \H_\seqb,&&&&
 \Phi_{\seq,\seqb,\seqc}:\H_{\seq,\seqb}\otimes \H_{\seqb,\seqc}\ra \H_{\seq,\seqc},\\
 &\Phi'_{\seq,\seqb}:\H'_\seq\otimes \H'_{\seq,\seqb}\ra \H'_\seqb&&\text{and}&&
 \Phi'_{\seq,\seqb,\seqc}:\H'_{\seq,\seqb}\otimes \H'_{\seqb,\seqc}\ra \H'_{\seq,\seqc},
 \end{align*}
where the first two maps are defined only if the rational replacements $K_\seq\leadsto K_\seqb\leadsto K_\seqc$ are orientation-preserving. Moreover, in the level of homology groups we have the equalities
\begin{align*}
&\Phi_{\seqb,\seqc}(\Phi_{\seq,\seqb}(\x\otimes\x_{\seq,\seqb})\otimes\x_{\seqb,\seqc})=
\Phi_{\seq,\seqc}(\x\otimes\Phi_{\seq,\seqb,\seqc}(\x_{\seq,\seqb}\otimes\x_{\seqb,\seqc})),\quad
\quad\forall\ \x\in \H_\seq,\ \x_{\seq,\seqb}\in \H_{\seq,\seqb},\ \x_{\seqb,\seqc}\in \H_{\seqb,\seqc},\\
&\Phi'_{\seqb,\seqc}(\Phi'_{\seq,\seqb}(\x\otimes\x_{\seq,\seqb})\otimes\x_{\seqb,\seqc})=
\Phi'_{\seq,\seqc}(\x\otimes\Phi'_{\seq,\seqb,\seqc}(\x_{\seq,\seqb}\otimes\x_{\seqb,\seqc})),\quad
\quad\forall\ \x\in \H'_\seq,\ \x_{\seq,\seqb}\in \H'_{\seq,\seqb},\ \x_{\seqb,\seqc}\in \H'_{\seqb,\seqc}.
\end{align*}
 
 When $\seq=\seqb$, both $\H_{\seq,\seqb}$ and $\H'_{\seq,\seqb}$ are isomorphic to $V^{g+1}$, where $V=\F[\var]\oplus\F[\var]$ is generated by a top generator $\theta_v$ and a bottom generator $\theta'_v$ (with respect to the homological grading). This gives the unique top classes $\theta_\seq\in\H_{\seq,\seq}$ and $\theta'_{\seq}\in\H'_{\seq,\seq}$ (c.f. \cite[Section 6.2]{AE-tangles}). Moreover, 
\[\Phi_\seq=\Phi_{\seq,\seq}(\cdot\otimes\theta_\seq):\H_\seq\ra \H_\seq\quad\text{and}\quad \Phi'_\seq=\Phi'_{\seq,\seq}(\cdot\otimes\theta'_\seq):\H'_\seq\ra \H'_\seq\]

 \begin{lem}\label{lem:top-generator}
 For every $\seq\neq \seqb$ as above, there are classes  $\theta'_{\seq,\seqb}\in \H'_{\seq,\seqb}$ and $\theta'_{\seqb,\seq}\in \H'_{\seqb,\seq}$ such that  \[\Phi'_{\seq,\seqb,\seq}(\theta'_{\seq,\seqb}\otimes\theta'_{\seqb,\seq})=\var^i\cdot\theta_{\seq}\quad\text{and}\quad \Phi'_{\seqb,\seq,\seqb}(\theta'_{\seqb,\seq}\otimes\theta'_{\seq,\seqb})=\var^i\cdot\theta_{\seqb}\]
for some $i\in\{0,1\}$. Moreover, if the rational replacement $K_\seq\leadsto K_\seqb$ is  orientation-preserving, there are classes  $\theta_{\seq,\seqb}\in \H_{\seq,\seqb}$ and $\theta_{\seqb,\seq}\in \H_{\seqb,\seq}$ such that for some $i\in\{0,1\}$ 
\[\Phi_{\seq,\seqb,\seq}(\theta_{\seq,\seqb}\otimes\theta_{\seqb,\seq})=\var^i\cdot\theta_{\seq}\quad\text{and}\quad \Phi_{\seqb,\seq,\seqb}(\theta_{\seqb,\seq}\otimes\theta_{\seq,\seqb})=\var^i\cdot\theta_{\seqb}\] 
 \end{lem} 

Lemma~\ref{lem:top-generator}, which is proved in the following section, implies the following theorem.

\begin{thm}\label{thm:RTR}
If the marked link $(K',\ps')$ is obtained from the marked link $(K,\ps)$ by a single rational replacement away from the markings, there are $\F[\var]$-homomorphisms 
\[\phi:\H'_{K,\ps}\ra \H'_{K',\ps'}\quad\text{and}\quad \phi': \H'_{K',\ps'}\ra \H'_{K,\ps}\quad\text{with}\quad \phi\circ \phi'=\var\cdot Id_{\H'_{K',\ps'}}\quad\text{and}\quad
\phi'\circ \phi=\var\cdot Id_{\H'_{K,\ps}}.\]
If the links are oriented and the replacement is an ORR, there are also $\F[\var]$-homomorphisms 
\[\psi:\H_{K,\ps}\ra \H_{K',\ps'}\quad\text{and}\quad \psi': \H_{K',\ps'}\ra \H_{K,\ps}\quad\text{with}\quad \psi\circ \psi'=\var\cdot Id_{\H_{K',\ps'}}\quad\text{and}\quad
\psi'\circ \psi=\var\cdot Id_{\H_{K,\ps}}.\]
\end{thm}

\section{Top generators in some special Heegaard diagrams}\label{sec:top-generators}
This section is devoted to the proof of Lemma~\ref{lem:top-generator}. Fix a Heegaard triple $(\Sig,\gammas_\seq,\gammas_{\seqb},\gammas_{\seqc},\zs,\ws)$. Later, we will further assume that $\seqc=\seq$ (i.e. $\gammas_\seqc=\gammas_\seq'$ is just a small Hamiltonian isotope of $\gammas_\seq$). Let $S$ denote the sphere component containing $\beta_0$ in the surface obtained by cutting $\Sigma$ along the curves in $\betas\setminus\{\beta_{0}\}$ and gluing disks to the resulting boundary components.  Four of the marked points, two from $\zs$ and two from $\ws$, are in $S$. We may label these marked points $\zs_0=\{z_1,z_2\}\subset\zs$ and $\ws_0=\{w_1,w_2\}\subset\ws$. The diagram $(S,\beta_0,\mu_0,\zs_0,\ws_0)$ is illustrated in Figure~\ref{fig:tangle-diagram} (bottom-left). Then $(S,\gamma_\seq,\gamma_\seqb,\gamma_\seqc,\zs_0\cup\ws_0)$  gives the chain complexes $E'_{\seq,\seqb},E'_{\seqb,\seqc}$ and $E'_{\seq,\seqc}$ with coefficients in $\A'$, as well as the triangle map
\[\Psi'_{\seq,\seqb,\seqc}:E'_{\seq,\seqb}\otimes E'_{\seqb,\seqc}\ra E'_{\seq,\seqc}.\]
If the RRs $K_\seq\leadsto K_\seqb\leadsto K_\seqc$ are orientation-preserving, the diagram $(S,\gamma_\seq,\gamma_\seqb,\gamma_\seqc,\zs_0,\ws_0)$ determines the chain complexes 
$E_{\seq,\seqb}$, $E_{\seqb,\seqc}$ and $E_{\seq,\seqc}$ with coefficients in $\A$, and the maps
\[\Psi_{\seq,\seqb,\seqc}:E_{\seq,\seqb}\otimes E_{\seqb,\seqc}\ra E_{\seq,\seqc}.\] 
Moreover, by choosing the almost complex structure  appropriately (see the argument of \cite[Proposition 5.1]{AE-cobordisms}), we may assume that 
\[C_{\star,\bullet}=E_{\star,\bullet}\otimes_{\F[\var]}V^{g+|\ps|-1}\quad\text{and}\quad C'_{\star,\bullet}=E'_{\star,\bullet}\otimes_{\F[\var]}V^{g+|\ps|-1}\quad\text{for}\ (\star,\bullet)\in\big\{(\seq,\seqb),(\seqb,\seqc),(\seq,\seqc)\big\},\]
while under these identifications $\Phi_{\seq,\seqb,\seqc}=\Psi_{\seq,\seqb,\seqc}\otimes Id$ and $\Phi'_{\seq,\seqb,\seqc}=\Psi'_{\seq,\seqb,\seqc}\otimes Id$. The proof of Lemma~\ref{lem:top-generator} is thus reduced to the following lemma, about diagrams on a sphere.

 \begin{lem}\label{lem:top-generator-2}
For $\seq,\seqb$ as above, there are closed classes  $\thetabarp_{\seq,\seqb}\in E'_{\seq,\seqb}$ and $\thetabarp_{\seqb,\seq}\in E'_{\seqb,\seq}$ such that  
\[\Psi'_{\seq,\seqb,\seq}(\thetabarp_{\seq,\seqb}\otimes\thetabarp_{\seqb,\seq})=\var^i\cdot\thetabar_{\seq},\quad\quad\text{for some } i\in\{0,1\},\] 
where $\thetabar_\seq$ denotes the top generator for $(S,\gamma_\seq,\gamma'_\seq,\zs_0\cup\ws_0)$.  Moreover, if $K_\seq\leadsto K_\seqb$ is orientation-preserving, there are closed classes  $\thetabar_{\seq,\seqb}\in E_{\seq,\seqb}$ and $\thetabar_{\seqb,\seq}\in E_{\seqb,\seq}$ such that  
\[\Psi_{\seq,\seqb,\seq}(\thetabar_{\seq,\seqb}\otimes\thetabar_{\seqb,\seq})=\var^i\cdot\thetabar_{\seq},\quad\quad\text{for some } i\in\{0,1\}.\] 
 \end{lem} 
 \begin{proof}
We prove the lemma in the case where $\seq=0$ (the general case is proved similarly).  We set $\beta_0=\gamma_\seq$, and note that $\gamma_0=\gamma_\seqb$ is obtained by applying $\map_\seqb$ to $\beta_0$, while $\delta_0=\gamma_\seqc$ is a small Hamiltonian isotope of $\beta_0$. The Heegaard diagram $(S,\beta_0,\gamma_0,\delta_0,\zs_0\cup\ws_0)$ is of the form described below. Let $R_1=\Dbb$ denote the unit disk in the complex plane and for $k\in\Z$, let  $x_k=\exp(\pi i k/n)\in\R_1$  denote $2n$ points on the boundary of $R_1$ for some integer $n$ (note that $x_{k+2n}=x_k$). For $k=1,\ldots,n$, connect $x_{k}$ to $x_{2n+1-k}$ using the path 
 \[\epsilon_k=\Big\{re^{\frac{\pi i k}{n}}\ \big|\ 0\leq 1-r\leq \frac{k}{2n}\Big\}\cup\Big\{\Big(1-\frac{k}{2n}\Big)e^{\frac{\pi i t}{n}}\ \big|\ t\in[1-k,k]\Big\}\cup\Big\{re^{\frac{\pi i(1-k)}{n}}\ \big|\ 0\leq 1-r\leq \frac{k}{2n}\Big\}.\]
 Place a marked point $w_1$ at $0$ (i.e. center of  $R_1$) and a marked point $z_1$ at $r\exp(\pi i/(2n))$ for some $0<r<1$ which is sufficiently close to $1$. The disk $R_1$ then contains two marked points, and $n$ arcs $\epsilon_1,\ldots,\epsilon_n$ with legs on its  boundary. Let  $R_2$ denote another copy of $R_1$ with reverse orientation, and with $x_k,\epsilon_k,z_1$ and $w_1$  renamed $y_k,\delta_k,w_2$ and  $z_2$, respectively.  Finally, choose an integer $\ell$ and identify the boundaries of $R_1$ and $R_2$ so that $y_{k}$ glues to $x_{\ell+k}$.  Let $\beta_0'$ denote the closed curve which is the common boundary of $R_1$ and $R_2$. If $\ell$ and $n$ are relatively prime, it follows that the union of $\epsilon_1,\ldots,\epsilon_n,\delta_1,\ldots,\delta_n$ is a simple closed curve $\gamma_0'$.  Every Heegaard diagram $(S,\beta_0,\gamma_0,\zs_0,\ws_0)$ is equivalent (isotopic) to one of
 \[H_{n,\ell}=\big(S'=R_1\cup R_2,\beta_0',\gamma_0',\zs_0,\ws_0\big).\]
  for some relatively prime integers $(\ell,n)$, unless $\gamma_0$ is isotopic to $\beta_0$ in $S-\zs_0\cup\ws_0$. We abuse the notation and make this identification. The diagram $H_{7,3}$ is illustrated in Figure~\ref{fig:torus-diagram-3}. Denote the chain complexes associated with $(S,\beta_0,\gamma_0,\zs_0,\ws_0)$ and the algebras $\A$ and $\A'$ by $(E_{0,\seqb},\partial_{0,\seqb})$ and $(E'_{0,\seqb},\partial'_{0,\seqb})$ respectively, while the complexes associated with $(S,\gamma_0,\delta_0,\zs_0,\ws_0)$ and the algebras $\A$ and $\A'$ are denoted by $(E_{\seqb,0},\partial_{\seqb,0})$ and $(E'_{\seqb,0},\partial'_{\seqb,0})$, respectively. In order for $E_{0,\seq}$ and $E_{\seq,0}$ to be chain complexes, it is necessary that $w_1$ and $w_2$ are separated by $\gamma_0$, which is the case if and only if $K=K_\seq\leadsto K_\seqb$ is orientation-preserving.\\

\begin{figure}
\def\svgwidth{0.85\textwidth}
{\small{
\begin{center}
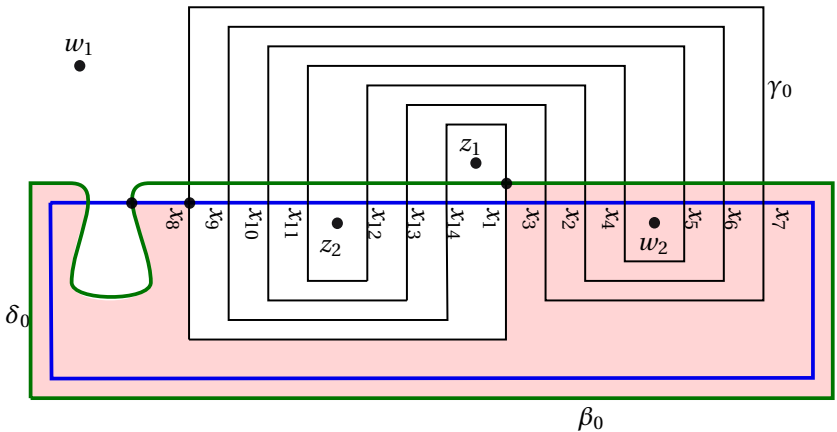 
\end{center}}}
\caption{\label{fig:torus-diagram-3} The Heegaard diagram $H_{7,4}$ is illustrated. The domain of a holomorphic triangle in $\pi_2(x_8,x'_1,\theta_0)$ is shaded.}
\end{figure} 


The domain of every Whitney disk which contributes to $\partial'_{\seqb,0}$ or $\partial'_{0,\seqb}$ is the union of one of the bigons containing $z_1,z_2,w_1$ or $w_2$, with some of the rectangles in the diagram.  Denote the chain complexes associated with $(S,\beta_0,\gamma_0,z_2,w_1)$ and $(S,\gamma_0,\delta_0,z_2,w_1)$ by 
\[(E,d_E)=\CFT(S,\beta_0,\gamma_0,z_2,w_1)\otimes_{\F[\var,\varw]}\A''\quad\text{and}\quad (F,d_F)=\CFT(S,\gamma_0,\delta_0,z_2,w_1)\otimes_{\F[\var,\varw]}\A'',\]
respectively, where $\A''=\F$ and the action of both $\var$ and $\varw$ on $\A''$ is multiplication by $1$. The chain complexes $E$ and $E'_{0,\seqb}$ have the same set of generators. However, the above observation implies that for every such generator $\x$ we have $\partial'_{0,\seq}(\x)=\var\cdot d_E(\x)$. Removing the marked points $z_1$ and $w_2$ allows us to change  $(S,\beta_0,\gamma_0,z_2,w_1)$  to  $(S,\beta_0,\delta_0,z_2,w_1)$ by isotopy. Therefore,  $E=\langle\x_1,\ldots,\x_{2n}\rangle_{\A''}$ for some generators $\x_i$, with  $d_E(\x_i)=0$ if $i$ is even or $i=1$, and $d_E(\x_{2i-1})=\x_{2i}$ for $i=2,\ldots,n$. Moreover, the homological degree of $\x_1$ is one more that the homological degree of $\x_2$. It thus follows that 
$\x_1,\ldots,\x_{2n}$ generate $E'_{0,\seqb}$ over $\A'$ and that
\begin{align*}
\partial'_{0,\seq}(\x_i)=\begin{cases}0&\text{if}\ i\ \text{is even or }i=1\\
\var\cdot\x_{i+1}&\text{otherwise}
\end{cases}.
\end{align*}    
We then set $\thetabarp_{0,\seqb}=\x_1$. Similarly,  we may assume that $F=\langle\y_1,\ldots,\y_{2n}\rangle_{\A''}$,  $d_F(\y_i)=0$ if $i$ is even or $i=1$, $d_F(\y_{2i-1})=\y_{2i}$ for $i=2,\ldots,n$, and  the homological degree of $\y_1$ is one more that the homological degree of $\y_2$. Then $\y_1,\ldots,\y_{2n}$ generate $E'_{\seqb,0}$ over $\A'$ and $\partial'_{\seq,0}(\y_i)=0$ unless $i>1$ is odd when we have $\partial'_{\seq,0}(\y_i)=\var\cdot \y_{i+1}$. Again, we set $\thetabarp_{\seqb,0}=\y_1$. To complete the proof of the first part of lemma, we need to show that 
$\Psi'_{0,\seqb,0}(\x_1\otimes\y_1)=\var\cdot \thetabar_0$ and $\Psi'_{\seqb,0,\seqb}(\y_1\otimes\x_1)=\var\cdot \thetabar_\seqb$. We only prove the first statement as the proofs are similar. \\

Every  holomorphic triangle which contributes to $\Psi'_{0,\seqb,0}(\x_1\otimes\y_1)$ is in correspondence with a holomorphic triangle which contributes to $\Psi(\x_1\otimes\y_1)$, where 
\[\Psi:E\otimes F\ra {\CFT}(S,\beta_0,\delta_0,z_2,w_1)\otimes_{\F[\var,\varw]}\A''\]
is the map associated with the diagram $(S,\beta_0,\gamma_0,\delta_0,z_2,w_1)$. Note that $\Psi(\x_1\otimes\y_1)=\thetabar_0$. The domain $\Dcal(\phi)$ of every holomorphic triangle $\phi$ which contributes to $\Psi_{0,\seqb,0}(\x_1\otimes\y_1)$ is of the form illustrated in Figure~\ref{fig:torus-diagram-3}, in the following sense. The  illustrated domain  belongs to $\pi_2(x_8,x_1',\theta_0)$, where $x_i'\in \delta_0\cap\gamma_0$ corresponds to $x_i\in\beta_0\cap\gamma_0$, and is in correspondence with the domain of a Whitney disk in $\pi_2(x_8,x_1)$ which contributes to $d_E(x_3)$. More generally, for every contributing $\phi\in\pi_2(x_i,x_j',\thetabar_0)$,  $\Dcal(\phi)$  is obtained from $\Dcal(\phi')$ for some  $\phi'\pi_2(x_i,x_j)$ which contributes to $d_E$, by adding/removing some of the small domains bounded between the curves $\beta_0$ and $\delta_0$. As such, $n_{\zs_0}(\phi)+n_{\ws_0}(\phi)=1$. Therefore, 
\[\Psi'_{0,\seqb,0}(\x_1\otimes\y_1)=\var\cdot \Psi(\x_1\otimes\y_1)=\var\cdot \thetabar_0.\]

Let us now assume that the rational replacement $K_0=K_\seq\leadsto K_\seqb$ is  orientation-preserving.  In this case, either of the three curves $\beta_0$, $\gamma_0$ and $\delta_0$ separates $z_1$ from $z_2$ and separates $w_1$ from $w_2$.  Therefore, $(S,\beta_0,\gamma_0,w_1,w_2)$ and  $(S,\gamma_0,\delta_0,w_1,w_2)$ are both admissible Heegaard diagrams for the same sutured manifold, which is also determined by $(S,\beta_0,\delta_0,w_1,w_2)$.  This time we let $(E,d_E)$ denote the chain complex $\widehat{\CFT}(S,\beta_0,\gamma_0,w_1,w_2)$ and $(F,d_F)$ denote the chain complex $\widehat{\CFT}(S,\gamma_0,\delta_0,w_1,w_2)$. The chain complexes $E$ and $E_{0,\seqb}$ have the same set of generators and for every generator $\x$ of $E_{0,\seqb}$, we have $\partial_{0,\seq}(\x)=\var\cdot d_E(\x)$. As discussed earlier, this implies that $E_{0,\seqb}$ includes a top generator $\thetabar_{0,\seqb}$. Similarly, $E_{\seqb,0}$ includes a top generator $\thetabar_{\seqb,0}$. Moreover, if $\Psi$ denotes the triangle map associated with the punctured Heegaard triple $(S,\beta_0,\gamma_0,\delta_0,w_1,w_2)$, it follows  that 
\[\Psi_{0,\seqb,0}(\thetabar_{0,\seqb}\otimes\thetabar_{\seqb,0})=\var\cdot\Psi (\thetabar_{0,\seqb}\otimes\thetabar_{\seqb,0})=\var\cdot \thetabar_0.\] 
 This completes the proof of the lemma.
\end{proof}
\section{The torsion invariants and their basic properties}\label{sec:properties}
Let us assume that $K$ is an oriented link and that $\ps$ is a marking of $K$. Let $H=(\Sig,\alphas,\betas,\zs,\ws)$ denote a corresponding Heegaard diagram. Let $\var$ act on $\F$ by multiplication by $1$, giving $\F$ the structure of a $\F[\var]$-module. Then $C_{K,\ps}\otimes_{\F[\var]}\F$ is  identified with $\widehat{\CFT}(\Sig,\alphas,\betas,\ws)$ and $C'_K\otimes_{\F[\var]}\F$ is  identified with $\CFT^\infty(\Sig,\alphas,\betas,\ws)\otimes_{F[\var]}\F$. In particular, 
\begin{align*}
&(\F\oplus\F)^{|\ps|-1}=H_*(\widehat{\CFT}(\Sig,\alphas,\betas,\ws))=H_*(C_{K,\ps}\otimes_{\F[\var]}\F)=\H_{K,\ps}\otimes_{\F[\var]}\F\quad\quad\text{and}\\
&(\F\oplus\F)^{|\ps|-1}=H_*(\CFT^\infty(\Sig,\alphas,\betas,\ws)\otimes_{\F[\var]}\F)=H_*(C'_{K,\ps}\otimes_{\F[\var]}\F)=\H'_{K,\ps}\otimes_{\F[\var]}\F.
\end{align*}
Let $|K|$ denote the number of connected components of $K$. For some sequences of positive integers $\ns_K=(\tf(K)=n_1\geq  \cdots\geq n_k>0)$ and $\ms_K=(\torsion(K)=m_1\geq\cdots\geq m_{k'}>0)$ we then have
\begin{align*}
&\H_{K,\ps}=\Big(\left(\F[\var]\oplus\F[\var]\right)^{|K|-1}\oplus \bigoplus_{i=1}^k\frac{\F[\var]}{\langle \var^{n_i}\rangle}\Big)\otimes\left(\F[\var]\oplus\F[\var]\right)^{|\ps|-|K|}\quad\text{and}\\ 
&\H'_{K,\ps}=\Big(\left(\F[\var]\oplus\F[\var]\right)^{|K|-1}\oplus \bigoplus_{i=1}^{k'}\frac{\F[\var]}{\langle \var^{m_i}\rangle}\Big)\otimes\left(\F[\var]\oplus\F[\var]\right)^{|\ps|-|K|}.
\end{align*}
In other words, the sequences $\ns_K$ and $\ms_K$ (together with $|\ps|$ and $|K|$) determine $\H_{K,\ps}$ and $\H'_{K,\ps}$ respectively. In fact, since $\F$ is a field and $\F[\var]$ is a PID, the chain homotopy types of $(C_{K,\ps},d_{K,\ps})$ and $(C'_{K,\ps},d'_{K,\ps})$ are determined by $\ns_K$ and $\ms_K$, respectively. Moreover, note that the sequences $\ns_K$ and $\ms_K$ do not depend on $\ps$, and that $\ns_K$ is even independent of the orientation of $K$. One should of course note that we are dropping the homological grading from our discussion to simplify the discussions. \\  

For $\ns=(n_1\geq \cdots \geq n_k>0)$, set $|\ns|=k$ and $\ns-1=(n_1-1\geq \cdots \geq n_l-1)$, where $l$ is the largest index so that $n_l>1$. Define $\ns-(p+1)$ recursively by $(\ns-p)-1$ for $p\geq 1$. We also set
\[\ns\cdot p:=\big(\underbrace{n_1,n_1,\ldots,n_1}_{p\ \text{times}}\geq\underbrace{n_2,n_2,\ldots,n_2}_{p\ \text{times}}\geq\cdots\geq\underbrace{n_k,n_k,\ldots,n_k}_{p\ \text{times}}\big).\]
For $\ns=(n_1\geq \cdots \geq n_k>0)$ and $\ns'=(n'_1\geq \cdots \geq n'_{k'}>0)$ we write $\ns\geq \ns'$ if $k\geq k'$ and $n_i\geq n_i'$ for all $i=1,\ldots,k$. Define $d(\ns,\ns')=\ell$  if $\ell$ is the smallest integer with $\ns\geq \ns'-\ell$ and $\ns'\geq \ns-\ell$.\\

 If $K_i$ is obtained from $K_{i-1}$ by an RR for $i=1,\ldots,\ell$, set $\ms_i=\ms_{K_i}\cdot 2^{\ell-|K_i|}$, where $\ell$ is the largest of $|K_i|$. If each $K_i$ is oriented and the RRs are orientation-preserving, we also set $\ns_i=\ns_{K_i}\cdot 2^{\ell-|K_i|}$. We equip each $K_i$ with a marking $\ps_i$ with $|\ps_i|=\ell$ so that $(K_i,\ps_i)$ is obtained from $(K_{i-1},\ps_{i-1})$ by an RR. Theorem~\ref{thm:RTR} implies that there are homomorphisms
\[\phi_i:\H'_{K_{i-1},\ps_{i-1}}\ra \H'_{K_i,\ps_i}\quad\text{and}\quad \psi_i:\H'_{K_i,\ps_i}\ra \H'_{K_{i-1},\ps_{i-1}}\quad\text{for}\ \ i=1,\ldots,\ell, \]
such that $\phi_i\circ\psi_i=\var$ and $\psi_i\circ\phi_i=\var$. Therefore, $\ms_i\geq \ms_{i-1}-1$ and $\ms_{i-1}\geq \ms_i-1$ for $i=1,\ldots,\ell$. If $K=K_0$ and $K'=K_\ell$,  these inequalities together imply 
\[\ms_K\cdot 2^{|K'|}\geq \ms_{K'}\cdot 2^{|K|}-\ell\quad\text{and}\quad \ms_{K'}\cdot 2^{|K|}\geq \ms_K\cdot 2^{|K'|}-\ell.\]
If each $K_i$ is obtained from $K_{i-1}$ by an ORR, a similar conclusion is obtained for $\ns_K$ and $\ns_{K'}$. As a consequence, we have the following proposition.

\begin{prop}\label{prop:distance-bounds}
For any pair of (oriented) links $K,K'$, we have the inequalities
\begin{align*}
&\torsion(K,K')=d\left(\ms_K\cdot 2^{|K'|},\ms_{K'}\cdot 2^{|K|}\right)\leq u'_q(K,K')\quad\quad\text{and}\\
&\tf(K,K')=d\left(\ns_K\cdot 2^{|K'|},\ns_{K'}\cdot 2^{|K|}\right)\leq u_q''(K,K')\leq u_q(K,K').
\end{align*}
In particular, $\torsion(K)\leq u'_q(K)$ and $\tf(K)\leq u_q(K)$ .  
\end{prop}  

If $K\#K'$ denotes the connected sum of two links $K$ and $K'$ (where a distinguished component of each link is connected to the other link), a Heegaard diagram for $K\#K'$ may be constructed by taking the connected sum of Heegaard diagrams for $K$ and $K'$ in an appropriate sense. This implies that  $\CFT(K\#K')$ is chain homotopy equivalent to  $\CFT(K)\otimes_{\F[\var,\varw]} \CFT(K')$. The K\"uneth formula implies the following corollary. 

\begin{cor}\label{cor:connected-sum}
For every two oriented links $K$ and $K'$, we have
\[\tf(K\#K')=\max\{\tf(K),\tf(K')\}\quad\text{and}\quad\torsion(K\#K')=\max\{\torsion(K),\torsion(K')\}.\]
\end{cor}

Suppose that $K$ has thin Floer homology, $\ns_K=(n_1\geq \cdots\geq n_k)$ and $\ms_K=(m_1\geq \cdots\geq m_{k'})$. Then $n_1=\cdots=n_k=m_1=\cdots=m_{k'}=1$. Therefore, we have the following corollary:

\begin{cor}\label{cor:alternation}
The rational distance of a link $K$ from the space $\Qcal$ of links with thin link Floer homology is at least $\torsion(K)-1$, while its OR-distance from $\Qcal$ is at least $\tf(K)-1$. In particular, the rational distance and the OR-distance of $K$ from the set of quasi-alternating knots is bounded below by $\torsion(K)-1$ and $\tf(K)-1$, respectively.
\end{cor}

Let  the marking $\ps$ of a knot $K$ consist of a single marked point. Since $\F$ is a field,  $C_{K,\ps}$ is filtered chain homotopy equivalent to a complex  generated by $\x_0,\ldots,\x_{2n}$ in Alexander gradings $s_0, \ldots, s_{2n}$ and homological gradings $\mu_0, \ldots,\mu_{2n}$ respectively, with the differential 
\[d_{K,\ps}(\x_i)=\begin{cases} 0&\text{if}\ i\ \text{is even}\\
\var^{s_i-s_{i+1}}\cdot \x_{i+1}&\text{if}\ i\ \text{is odd}
\end{cases},\]
while $s_i-s_{i+1}>0$ for all odd values of $i$ (i.e. we start with the $E_2$ term of the corresponding spectral sequence). In particular, $\mu_{i+1}=\mu_i-1$ if $i$ is odd, and $\mu_0=0$.  Therefore,
\begin{align*}
\tau(K)=s_0\quad\quad\text{and}\quad\quad \tf(K)=\max\big\{s_{2j-1}-s_{2j}\ \big|\ j=1,2,\ldots,n\big\},
\end{align*}
where $\tau(K)$ denotes the Ozsv\'ath-Szab\'o tau invariant \cite{OS-four-ball}. Let us now assume that $\tf(K)=1$. It then follows that $s_{2j-1}-s_{2j}=\mu_{2j-1}-\mu_{2j}=1$ for $j=1,\ldots,n$. Therefore,
\begin{align*}
Q_K(q,t)&:=\sum_{i,j} \mathrm{dim}\left(\widehat{\mathrm{HFK}}_j(K,i)\right)\cdot q^jt^i=q^{\mu_0}t^{\tau(K)}+(1+qt)\cdot\sum_{j=1}^n q^{\mu_{2j}}t^{s_{2j}},
\end{align*}
which is of the form $t^{\tau(K)}+(1+qt)\cdot P_K(q,t)$. This observation implies the following corollary.

\begin{cor}\label{cor:Alexander-polynomial}
If $\tf(K)=1$ for a knot $K$, then  $Q_K(q,t)-t^{\tau(K)}$ is divisible by $1+qt$ and all the coefficients of $(Q_K(q,t)-t^{\tau(K)})/(1+qt)$ are non-negative integers. 
\end{cor} 

\section{Examples and sample computations}\label{sec:examples}
In this section, we examine the bounds constructed in the previous sections in a number of examples. The first example is of course the case of the torus knots.

\begin{example}\label{ex:torus-knot}
Let $K=T_{p,q}$ be the $(p,q)$ torus knot with $0<p<q$. Assume that 
\[\frac{(t^{pq}-1)(t-1)}{(t^p-1)(t^q-1)}=\sum_{i=0}^{2n}(-1)^it^{a_i}\]
for a sequence $(p-1)(q-1)=a_0>a_1>...>a_{2n}=0$ of integers. Let us define $b_i=a_{i-1}-a_{i}$ for $i=1,\ldots,2n$ and set $\bs=\bs_{p,q}=(b_1,\ldots,b_{2n})$. Note that $b_i$ are all positive integers. The knot Floer complex associated with $K$ is then determined by $\bs$ (\cite{OS-lspace} and \cite[Example 5.1]{AE-unknotting}). In particular, $(C_K,d_K)$ and $(C'_K,d'_K)$ are  freely generated over $\F[\var]$ by the generators $\{\x_i\}_{i=0}^{2n}$ and are equipped with the differential 
\[d_K(\x_i)=\begin{cases}
\var^{b_{i}}\x_{i-1}
\quad&\text{if}\ i\ \text{is odd}\\
0\quad &\text{if}\ i\ \text{is even}
\end{cases}\quad\text{and}\quad
d'_K(\x_i)=\begin{cases}
\var^{b_{i}}\x_{i-1}+\var^{b_{i+1}}\x_{i+1}
\quad&\text{if}\ i\ \text{is odd}\\
0\quad &\text{if}\ i\ \text{is even}.
\end{cases}\]
Therefore, $\H_K$ is generated by $\h_i=[\x_{2i}]$ for $i=0,\ldots,n$, with $\h_n$ free and $\h_i$ a torsion element of order $b_{2i+1}$ for $i=0,\ldots,n-1$. Moreover, $\H'_K$ is generated by $\h'_i=[\x_{2i}]$ for $i=0,\ldots,n$, while
\[\var^{b_{2i-1}}\h'_{i-1}=\var^{b_{2i}}\h'_i\quad\quad i=1,\ldots,n.\]
In particular, it follows that
\begin{align*}
\tf(K)=\max\big\{b_{2i+1}\ \big|\ i=1,\ldots,n\big\}\quad\text{and}\quad\torsion(K)=\max \big\{\min\left\{b_{2i-1},b_{2i}\right\}\ \big|\ i=1,\ldots,n\big\}. 
\end{align*}
Let us restrict our attention to the case where $q=pk+1$. In this case we have 
\begin{align*}
&\bs_{p,pk+1}=(\underbrace{1,p-1,\ldots,1,p-1}_{k\ \text{times}},\underbrace{2,p-2,\ldots,2,p-2}_{k\ \text{times}},\ldots,\underbrace{p-1,1,\ldots,p-1,1}_{k\ \text{times}})\\
\Rightarrow\quad
&\tf(T_{p,pk+1})=p-1\quad\text{and}\quad\torsion(T_{p,pk+1})=\max\big\{\min\big\{i,p-i\big\}\ \big|\ i=1,\ldots,p-1\big\}=\left\lfloor\frac{p}{2}\right\rfloor.
\end{align*}

\begin{figure}
\def\svgwidth{\textwidth}
\begin{center}
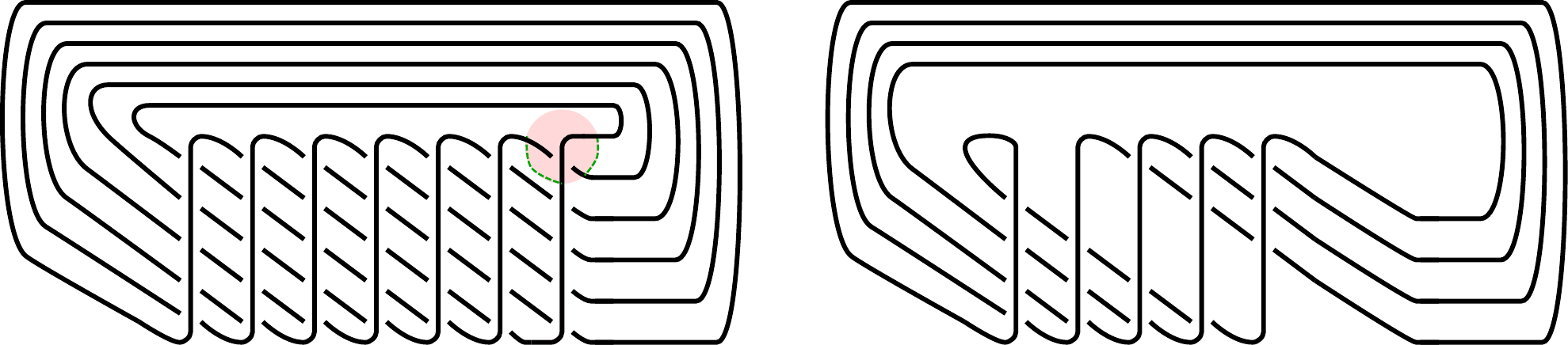
\caption{\label{fig:torus-knot} A single RR changes $T_{p,pk+1}$ (the left diagram) to $T_{p-2,(p-2)k+1}$ (the right diagram). The case $p=6$ and $k=1$ is illustrated.}
\end{center}
\end{figure}

\begin{cor}\label{torus-knot}
For the torus knot $T_{p,pk+1}$ we have \[u_q\left(T_{p,pk+1}\right)\geq u''_q\left(T_{p,pk+1}\right)\geq p-1\quad\text{and}\quad u'_q\left(T_{p,pk+1}\right)=\left\lfloor \frac{p}{2}\right\rfloor.\]
\end{cor}
\begin{proof}
The above observation implies that $u_q''(T_{p,pk+1})\geq p-1$ while  $u'_q(T_{p,pk+1})\geq \lfloor p/2\rfloor$. To see the equality in the latter inequality, note that $T_{p,pk+1}$ may be unknotted by $\lfloor p/2\rfloor$ RRs.  Figure~\ref{fig:torus-knot} illustrates how a single RR (which is in fact a resolution of one of the crossings) changes  $T_{p,pk+1}$ to $T_{p-2,(p-2)k+1}$. Therefore, by resolving $\lfloor p/2\rfloor$ crossings which are chosen appropriately, we arrive at the unknot.
\end{proof}

For other values of $q>p$,  $u'_q(T_{p,q})$ may be smaller than $\lfloor p/2\rfloor$. In fact, if we resolve one of the crossings in $T_{p,q}$ so that the connectivity is preserved, we obtain $T_{p',q'}$, where $p'$ and $q'$ are determined as follows. Suppose that $i$ is the least positive integer so that $iq=jp\pm 1$ for some positive integer $j$ (thus, $i\leq p/2$). Then $q'=q-2j$ and $p'=p-2i$. In this situation we write $(p,q)\leadsto (p',q')$. Let $k(p,q)$ denote the least number $k$ so that 
\[(p,q)=(p_0,q_0)\leadsto (p_1,q_1)\leadsto (p_2,q_2)\leadsto \cdots\leadsto (p_k,q_k),\]
where $p_k\in\{0,1\}$. This means that $T_{p,q}$ may be turned into the unknot by $k(p,q)$ crossing resolutions. In particular, $u_q(T_{p,q})$ is at most $k(p,q)\leq p/2$. 

\begin{cor}\label{cor:torus-knot-general}
If $1<p<q$ are relatively prime integers, we have   
\[\torsion(T_{p,q})\leq u'_q(T_{p,q})\leq k(p,q)\quad\quad\text{and}\quad\quad\tf(T_{p,q})=p-1.\]
\end{cor}
\begin{proof}
We have already proved the first claim. For the second claim, note that the number of terms in $\bs_{p,q}$ is always even and that the last two terms in $\bs_{p,q}$ are always $p-1$ and $1$. Moreover, no $b_i$ is greater than $p-1$. These observations suffice to give the second claim.
\end{proof}
\end{example}

\begin{figure}
\begin{center}
\includegraphics[scale=0.5]{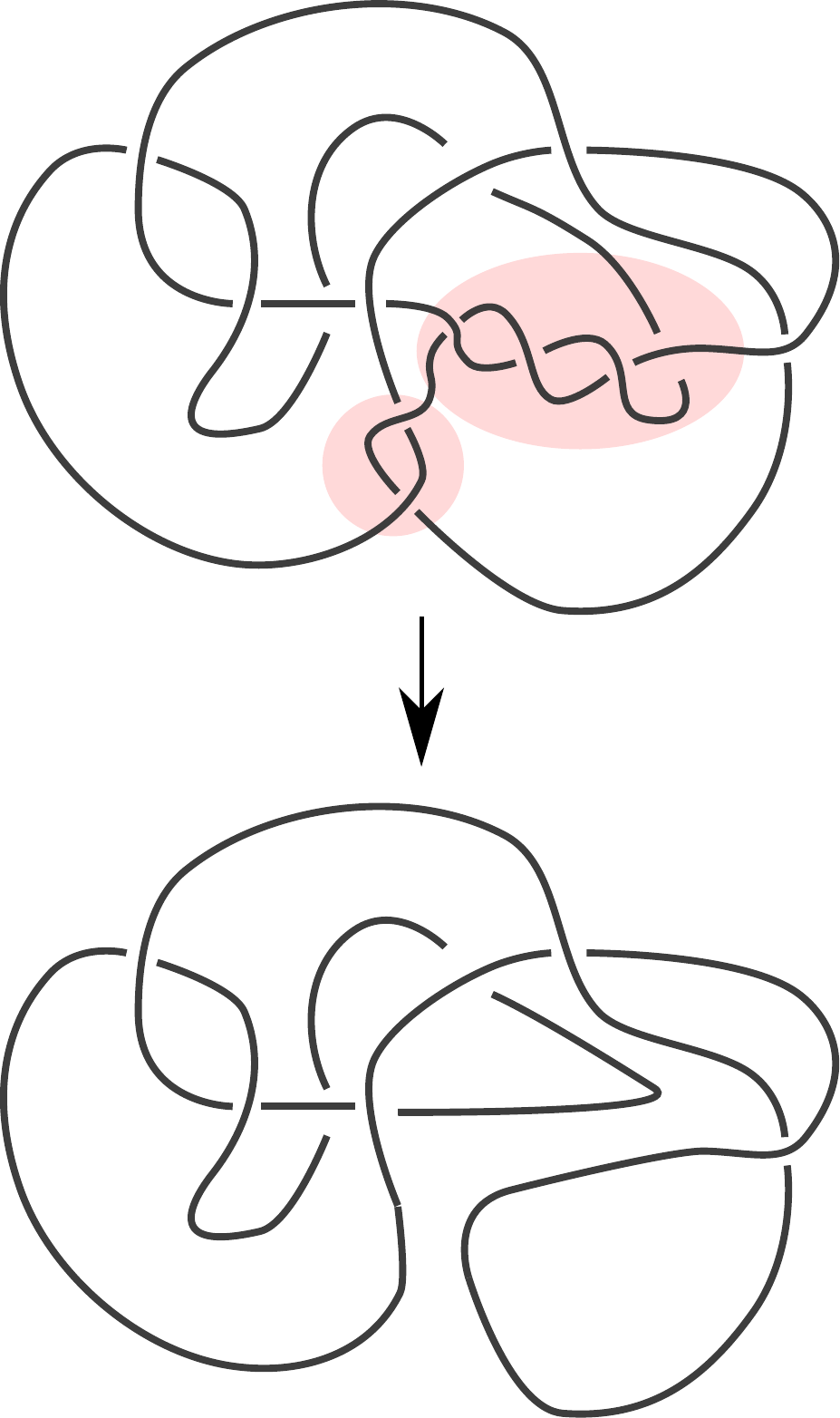}\ \ \ \ \ 
\includegraphics[scale=0.45]{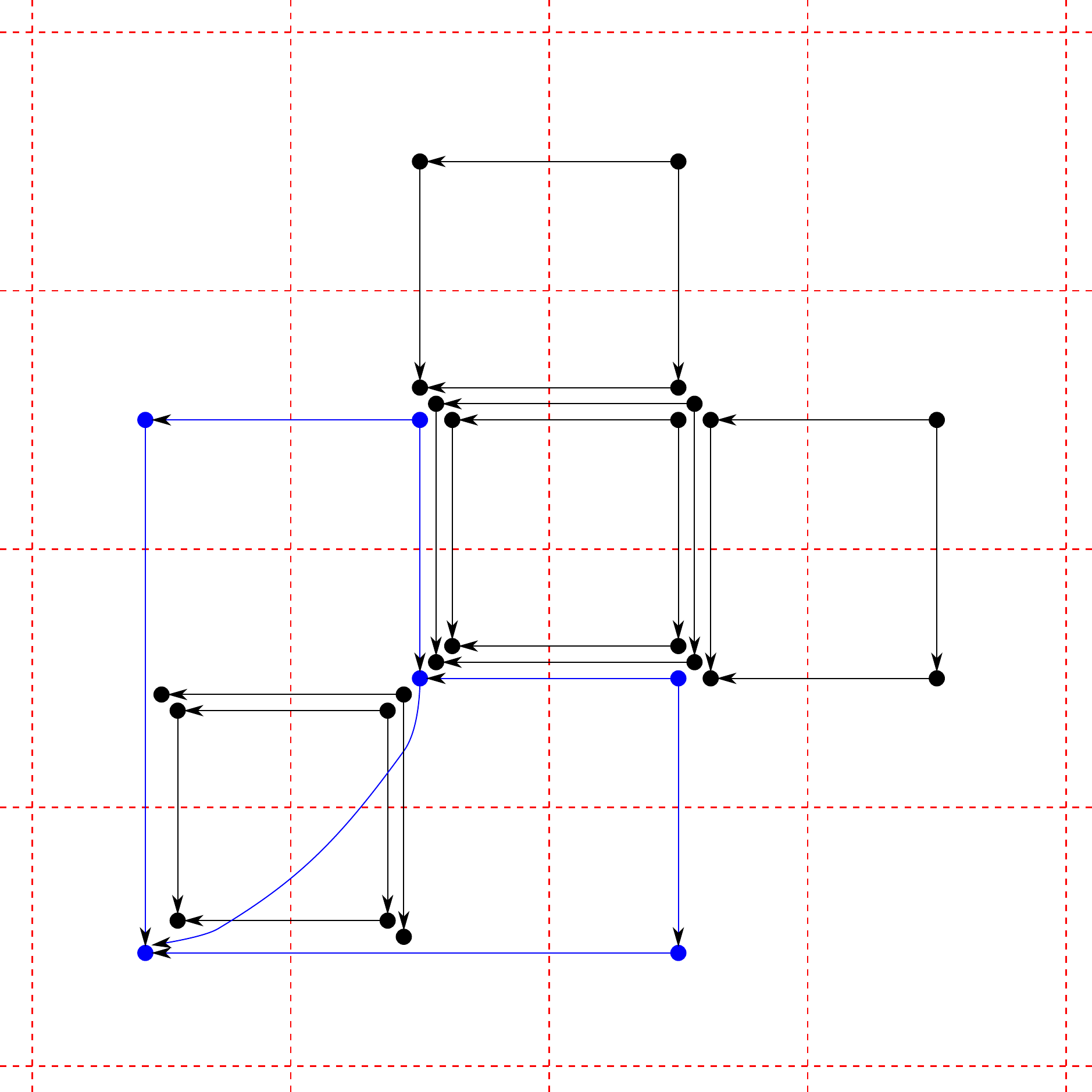}
\caption{
The knot $12n_{404}$ and the corresponding knot chain complex.
}\label{fig:12n404}
\end{center}
\end{figure}  

\begin{example}
The $(1,1)$ knot $K=12n_{404}$, which is illustrated on the left-hand-side of Figure~\ref{fig:12n404} is given by  the quadruple $[29,7,14,1]$ in Rasmussen's notation  \cite[page 14]{Ras-2}, and the corresponding knot chain complex  may be computed combinatorially (c.f. \cite[Example 5.4]{AE-unknotting}).  The right-hand-side of Figure~\ref{fig:12n404} describes the knot chain complex associated with $K$. Each dot in the diagram represents a generator. An arrow which connects a dot corresponding to a generator $\x$ to a dot representing a generator $\y$ and cuts $i$ vertical lines and $j$ horizontal lines corresponds to the contribution of $\var^i\varw^j\y$ to $d(\x)$. The blue dots generate a sub-complex which is more interesting for us. When we set $\varw=\var$ (to obtain the chain complex $C_K$) it follows that the homology of the sub-complex generated by the blue dots is $\F^2\oplus (\F[\var]/\langle \var^2\rangle)$. In fact, we may quickly compute 
\[\H'_K=\F^{13}\oplus\F[\var]\oplus\frac{\F[\var]}{\langle\var^2\rangle}.\]
In particular, $\torsion(12n_{404})=2$. Note that the PR-unknotting number of $12n_{404}$ is at most $2$. To see this, use the two balls determined by red disks in Figure~\ref{fig:12n404} (left), for rational  replacements, which give the unknot illustrated on the bottom-left of the aforementioned Figure. Since $\torsion(12n_{404})=2$, it follows that both the rational unknotting number and the PR-unknotting number of $12n_{404}$ are equal to $2=\torsion(12n_{404})=\tf(12n_{404})$. As discussed in \cite[Example 5.4]{AE-unknotting}, the unknotting number of $12n_{404}$ is not known, while it satisfies $2\leq u(12n_{404})\leq 3$.
\end{example}

\begin{example} The chain complexes associated with the $(2,-1)$ cable of $T_{2,3}$, which is denoted by $K=T_{2,3;2,-1}$ and the $(2,-3)$ cable of $T_{2,3}$, which is denoted by $K'=T_{2,3;2,-3}$ are studied in \cite{AE-unknotting} and are both illustrated in Figure~\ref{fig:cable-complexes}. It follows that \[u_q(K),u_q(K')\geq \tf(K)=\tf(K')=2\quad\text{and}\quad u_q'(K)=u_q'(K')=\torsion(K)=\torsion(K')=1.\] Note that the a single RR is sufficient for unknotting the $(2,2k+1)$-cable of any knot. 
\begin{figure}
\begin{center}
\def\svgwidth{0.45\textwidth}
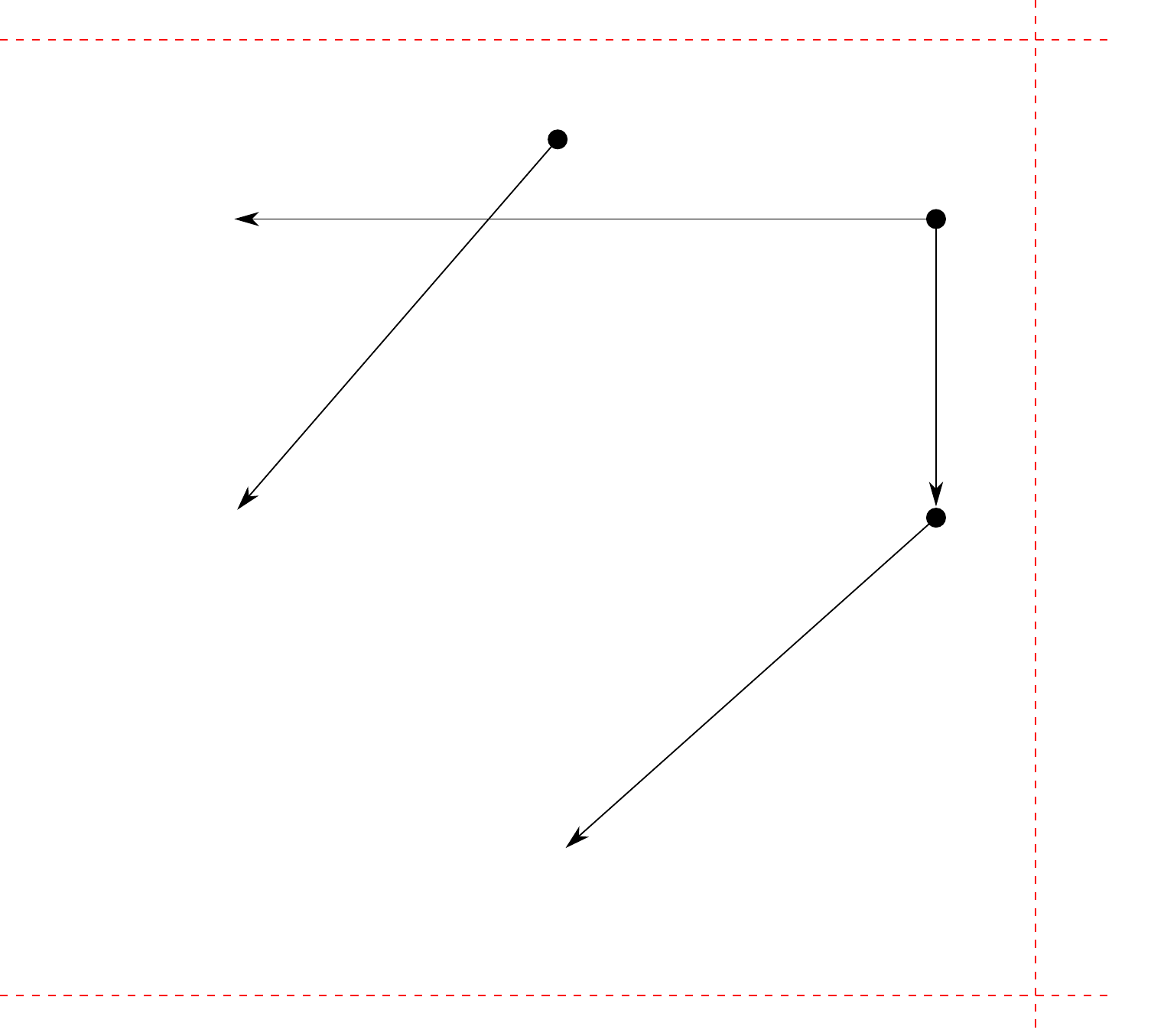
\def\svgwidth{0.45\textwidth}
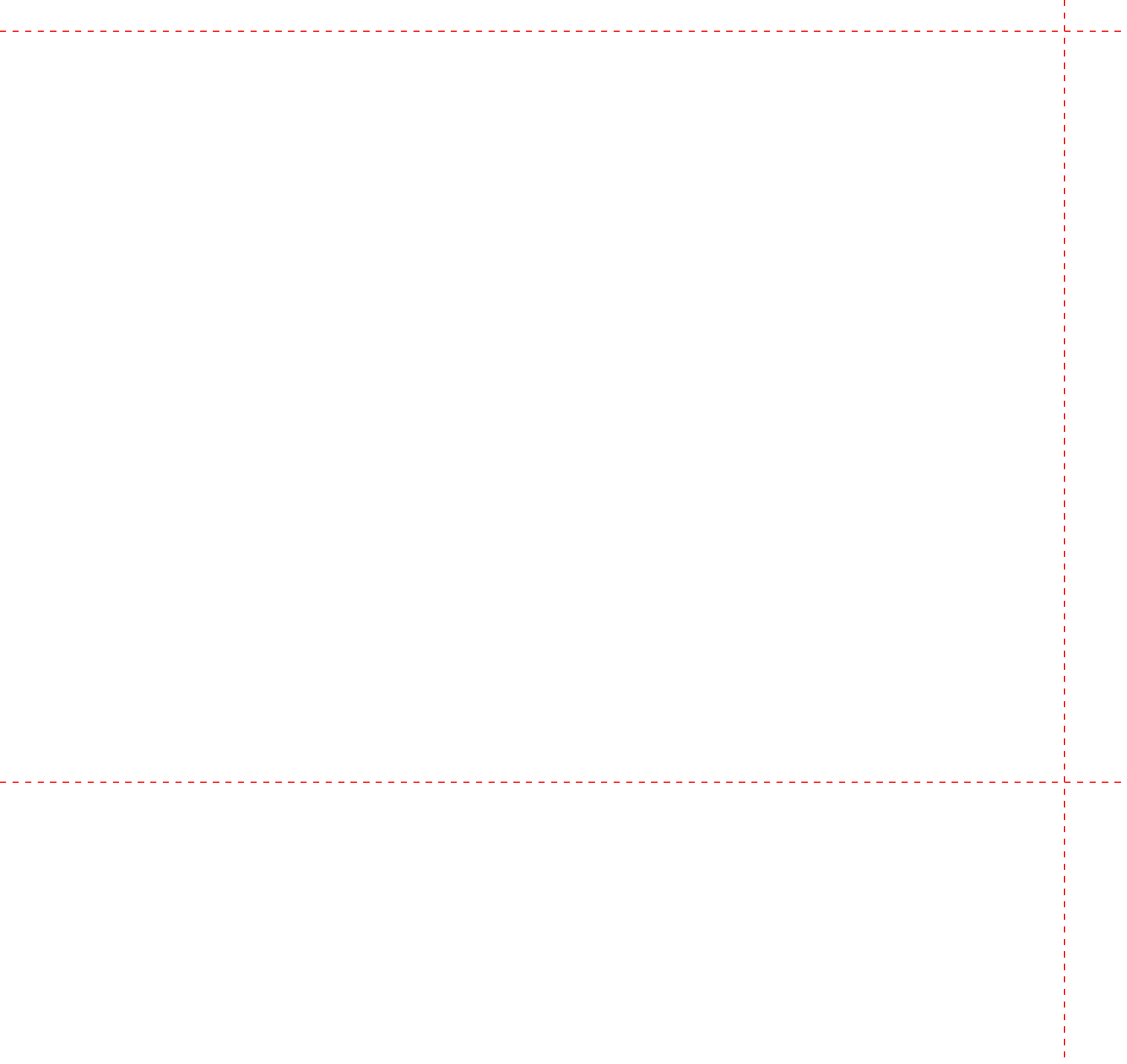
\caption{\label{fig:cable-complexes} The chain complexes associated with $T_{2,3;2,-1}$ (left) and $T_{2,3;2,-3}$ (right).}
\end{center}
\end{figure}

\end{example}

\begin{example}
According to {\sf{Knotinfo}} tables \cite{Knotinfo}, among all the knots with at most $10$ crossings the only knots which are not quasi-alternating are the ones in  
\[\Acal_{10}=\big\{8_{19}, 9_{42}, 9_{46},10_{124}, 10_{128}, 10_{132}, 10_{136},10_{139}, 10_{140}, 10_{145}, 10_{152}, 10_{153}, 10_{154}, 10_{161}\big\}.\]
For any non-trivial knot $K$ with at most $10$ crossings which is not in $\Acal_{10}$, we thus have $\torsion(K)=\tf(K)=1$. Moreover, the unknotting numbers of $9_{42}$, $10_{132}$ and $10_{136}$ are $1$ (c.f. \cite{Knotinfo}). So, for these three knots we have $\tf(K)=\torsion(K)=u(K)=1$, as well. Since $8_{19}=T_{3,4}$ and $10_{124}=T_{3,5}$, Corollary~\ref{cor:torus-knot-general} implies 
\[\tf(8_{19})=\tf(10_{124})=2\quad\quad\text{and}\quad\quad\torsion(8_{19})=\torsion(10_{124})=1.\]  One can check by hand that the remaining $9$ knots (i.e. $9_{46}$, $10_{128}$, $10_{139}$, $10_{140}$, $10_{145}$, $10_{152}$, $10_{153}$, $10_{154}$ and $10_{161}$) may be changed to an alternating knot with a single PRR. Therefore, for these latter $9$ knots  we have $\tf(K),\torsion(K)\leq 2$.  On the other hand, the value of the Ozsv\'ath-Szab\'o polynomial $Q_K(q,t)$ for these $9$ knots are listed in the following table (these computations are borrowed from \cite{BG-computations}):
\begin{equation}\label{eq:HFK-computations}
\arraycolsep=1.2pt\def\arraystretch{1.5}
\begin{array}{|c|c|c|}
\hline
\text{Name}&\tau(K)&Q_K(q,t)\\
\hline
9_{46} & 0 & 2q^{-1}t^{-1} + 5 + 2qt \\
\hline
10_{128} & 3 & 2q^{-6}t^{-3} + 3q^{-5}t^{-2} + q^{-4}t^{-1} + q^{-2} + q^{-2}t + 3q^{-1}t^{2} + 2t^{3}\\
\hline
10_{139} & 4 & q^{-8}t^{-4} + q^{-7}t^{-3} + 2q^{-4}t^{-1} + 3q^{-3} + 2q^{-2}t + q^{-1}t^{3} + t^{4}\\
\hline
10_{140} & 0 & q^{-2}t^{-2} + 2q^{-1}t^{-1} + 3 + 2qt + q^{2}t^{2}\\
\hline
10_{145} & -2 & t^{-2} + 2t^{-1} + qt^{-1} + 4q + q^{2} + 2q^{2}t + q^{3}t + q^{4}t^{2}\\
\hline
10_{152} & -4 & t^{-4} + qt^{-3} + qt^{-2} + 4q^{2}t^{-1} + 5q^{3} + 4q^{4}t + q^{5}t^{2} + q^{7}t^{3} + q^{8}t^{4}\\
\hline
10_{153} & 0 &q^{-2}t^{-3}(1+t+2qt+2qt^2) + t^{-1} + 3 + 2qt + q^{2}t + q^{2}t^{2} + 2q^{3}t^{2} + q^{4}t^{3}\\
\hline
10_{154} & 3 & q^{-6}t^{-3}(1+qt+q^2t) + 4q^{-3}t^{-1} + 7q^{-2} + 4q^{-1}t + q^{-1}t^2+t^{2} + t^{3} \\
\hline
10_{161} & -3 & t^{-3} + t^{-2}+ qt^{-2} + 2qt^{-1} + 3q^{2} + 2q^{3}t + q^{4} t^{2}+ q^{5}t^{2} + q^{6}t^{3}\\
\hline
\end{array}
\end{equation}
It follows that $9_{46}$ and $9_{140}$ have thin knot Floer homology (therefore, $\tf(K)=\torsion(K)=1$ for these two knots). For $K=10_{128}, 10_{139}, 10_{152}, 10_{154}, 10_{161}$, the polynomial $Q_K(q,t)-t^{\tau(K)}$ is not divisible by $1+qt$. Therefore, 
\[\tf(10_{128})=\tf(10_{139})=\tf(10_{152})=\tf(10_{154})=\tf(10_{161})=2.\]
For the remaining knots $K=10_{145}, 10_{153}$, it follows from the computation of $Q_K(q,t)$ that $\tf(K)$ is forced to be $1$. 
\end{example}
\bibliographystyle{hamsalpha}

\newcommand{\etalchar}[1]{$^{#1}$}
\begin{thebibliography}{CGL{\etalchar{+}}20}

\bibitem[AD19]{AD-unknotting}
Akram Alishahi and Nathan Dowlin.
\newblock The {L}ee spectral sequence, unknotting number, and the knight move
  conjecture.
\newblock {\em Topology Appl.}, 254:29--38, 2019.

\bibitem[AE15]{AE-tangles}
Akram~S. Alishahi and Eaman Eftekhary.
\newblock A refinement of sutured {F}loer homology.
\newblock {\em J. Symplectic Geom.}, 13(3):609--743, 2015.

\bibitem[AE16]{AE-cobordisms-v1}
Akram Alishahi and Eaman Eftekhary.
\newblock Tangle {F}loer homology and cobordisms between tangles.
\newblock {\em arXiv:1610.07122v1}, 2016.

\bibitem[AE20a]{AE-unknotting}
Akram Alishahi and Eaman Eftekhary.
\newblock Knot {F}loer homology and the unknotting number.
\newblock {\em Geom. Topol.}, 24(5):2435--2469, 2020.

\bibitem[AE20b]{AE-cobordisms}
Akram Alishahi and Eaman Eftekhary.
\newblock Tangle {F}loer homology and cobordisms between tangles.
\newblock {\em J. Topol.}, 13(4):1582--1657, 2020.

\bibitem[Ali19]{Alishahi-unknotting}
Akram Alishahi.
\newblock Unknotting number and {K}hovanov homology.
\newblock {\em Pacific J. Math.}, 301(1):15--29, 2019.

\bibitem[BG12]{BG-computations}
John~A. Baldwin and William~D. Gillam.
\newblock Computations of {H}eegaard-{F}loer knot homology.
\newblock {\em J. Knot Theory Ramifications}, 21(8):1250075, 65, 2012.

\bibitem[CGL{\etalchar{+}}20]{Cetal-unknotting}
Carmen Caprau, Nicolle Gonz\'alez, Christine Ruey~Shan Lee, Adam~M. Lowrance,
  Radmila Sazdanovi\'c, and Melissa Zhang.
\newblock On {K}hovanov homology and related invariants.
\newblock {\em arXiv:2002.05247}, 2020.

\bibitem[Con70]{Conway-tangles}
J.~H. Conway.
\newblock An enumeration of knots and links, and some of their algebraic
  properties.
\newblock In {\em Computational {P}roblems in {A}bstract {A}lgebra ({P}roc.
  {C}onf., {O}xford, 1967)}, pages 329--358. Pergamon, Oxford, 1970.

\bibitem[ILM21]{ILM-rational}
Damian Iltgen, Lukas Lewark, and Laura Marino.
\newblock {K}hovanov homology and rational unknotting.
\newblock {\em arXiv:2110.15107}, 2021.

\bibitem[KL04]{Kauffman-tangles}
Louis~H. Kauffman and Sofia Lambropoulou.
\newblock On the classification of rational tangles.
\newblock {\em Adv. in Appl. Math.}, 33(2):199--237, 2004.

\bibitem[Lin96]{lines}
Daniel Lines.
\newblock Knots with unknotting number one and generalised {C}asson invariant.
\newblock {\em J. Knot Theory Ramifications}, 5(1):87--100, 1996.

\bibitem[LM22]{Knotinfo}
Charles Livingston and Allison~H. Moore.
\newblock Knotinfo: Table of knot invariants.
\newblock URL: \url{knotinfo.math.indiana.edu}, March 2022.

\bibitem[McC15]{McCoy}
Duncan McCoy.
\newblock Non-integer surgery and branched double covers of alternating knots.
\newblock {\em J. Lond. Math. Soc. (2)}, 92(2):311--337, 2015.

\bibitem[MO08]{MO-quasi}
Ciprian Manolescu and Peter Ozsv\'{a}th.
\newblock On the {K}hovanov and knot {F}loer homologies of quasi-alternating
  links.
\newblock In {\em Proceedings of {G}\"{o}kova {G}eometry-{T}opology
  {C}onference 2007}, pages 60--81. G\"{o}kova Geometry/Topology Conference
  (GGT), G\"{o}kova, 2008.

\bibitem[MZ21]{MZ-rr}
D.~McCoy and R.~Zentner.
\newblock The montesinos trick for proper rational tangle replacement.
\newblock {\em preprint, arXiv:2110.15106}, 2021.

\bibitem[OS03a]{OS-alternating}
Peter Ozsv\'{a}th and Zolt\'{a}n Szab\'{o}.
\newblock {H}eegaard-{F}loer homology and alternating knots.
\newblock {\em Geom. Topol.}, 7:225--254, 2003.

\bibitem[OS03b]{OS-four-ball}
Peter Ozsv{\'a}th and Zolt{\'a}n Szab{\'o}.
\newblock Knot {F}loer homology and the four-ball genus.
\newblock {\em Geom. Topol.}, 7:615--639, 2003.

\bibitem[OS05]{OS-lspace}
Peter Ozsv\'{a}th and Zolt\'{a}n Szab\'{o}.
\newblock On knot {F}loer homology and lens space surgeries.
\newblock {\em Topology}, 44(6):1281--1300, 2005.

\bibitem[OS08]{OS-link}
Peter Ozsv\'{a}th and Zolt\'{a}n Szab\'{o}.
\newblock Holomorphic disks, link invariants and the multi-variable {A}lexander
  polynomial.
\newblock {\em Algebr. Geom. Topol.}, 8(2):615--692, 2008.

\bibitem[Ras05]{Ras-2}
Jacob Rasmussen.
\newblock Knot polynomials and knot homologies.
\newblock In {\em Geometry and topology of manifolds}, volume~47 of {\em Fields
  Inst. Commun.}, pages 261--280. Amer. Math. Soc., Providence, RI, 2005.

\end{thebibliography}
\newcommand{\etalchar}[1]{$^{#1}$}

\end{document}